\documentclass[12pt, a4paper, leqno]{amsart}

\usepackage[utf8]{inputenc}
\usepackage{amsmath} 
\usepackage{amsfonts}
\usepackage{amssymb}
\usepackage{txfonts}   
\usepackage{amsfonts}     
\usepackage{graphicx}     
\usepackage[dvipsnames]{xcolor}
\usepackage[colorlinks, allcolors=RedViolet]{hyperref}
\usepackage{cancel}

\newcommand{\Rn}{{\varmathbb{R}^n}}

\newcommand{\W}{\mathcal{W}}
\newcommand{\V}{\mathcal{V}}
\newcommand{\Ha}{\mathcal{H}} 
\newcommand{\M}{\mathcal{M}}

\newcommand{\I}{\mathcal{I}}

\def\diam{\qopname\relax o{diam}}
\def\loc{{\qopname\relax o{loc}}}

\def\min{\qopname\relax o{min}}
\def\diam{\qopname\relax o{diam}}
\def\inte{\qopname\relax o{int}}

\def\phi{\varphi}

\def\wL{{w\text{-}L}}

\let\oldmarginpar\marginpar
\renewcommand\marginpar[1]{\-\oldmarginpar[\raggedleft\footnotesize #1]%
{\raggedright\footnotesize #1}}

\usepackage{amsthm}

\swapnumbers
\theoremstyle{plain}
\newtheorem{theorem}[equation]{Theorem}
\newtheorem{lemma}[equation]{Lemma}

\theoremstyle{definition}

\theoremstyle{remark}
\newtheorem{remark}[equation]{Remark}

\numberwithin{equation}{section}

\title{A note on new mapping properties for Wolff potential}

\author{Petteri Harjulehto}
\address[Petteri Harjulehto]{Department of Mathematics and Statistics,
FI-00014 University of Helsinki, Finland}
\email{petteri.harjulehto@helsinki.fi}

\author{Ritva Hurri-Syrj\"anen}
\address[Ritva Hurri-Syrj\"anen]{Department of Mathematics and Statistics,
FI-00014 University of Helsinki, Finland}
\email{ritva.hurri-syrjanen@helsinki.fi}

\date{\today}

\begin{document}

\thanks{The authors are grateful to John L. Lewis for inspiring discussions about the early work of Tom Wolff on nonlinear potential theory.}


\begin{abstract} 
We study integrability properties of the Wolff potential in context of Choquet integrals with respect to the Hausdorff content.
As an application we give integrability results to the solutions of the $p$-Laplace equation in this context.
\end{abstract}

\maketitle

\section{Introduction}\label{Intro}

The Wolff
potential used in the present paper is given by
\begin{equation*}
  \W^{f}_{\alpha ,p}(x):=\int_0^\infty\left( \frac{\int_{B(x,r)} |f(y)| \, dy }{r^{n-\alpha p}}
  \right)^{1/(p-1)}\frac{dr}{r},
\end{equation*}
where
$n\geq 2$,
$p\in (1,\infty )$, 
 and  $\alpha p\in (0,n)$,
\cite[Section 4.5]{AH}.
Here $B(x,r)$ is a ball in $\Rn$ centred at $x$ with radius $r$ and $f$ is a locally integrable function defined on  $\Rn$.

The $W$-potential was introduced 
in \cite[p. 169]{HedbergWolff1983}
by Thomas H. Wolff  \cite[p. 166]{HedbergWolff1983} in his joint paper with Lars Inge Hedberg.
David R.\ Adams and Hedberg coined the name Wolff potential for the $W$-potential, 
\cite[Section 4.9]{AH}.
Based on the use of the $W$-potential it was possible to develop the nonlinear $L^p$ potential theory parallel to the classical $L^2$ theory.
The Wolff potential has turned out to be an essential tool to study problems in analysis.
In the present paper we consider the integrability properties of the Wolff potential in the context of Choquet integrals with respect to the Hausdorff content.
We benefit in our proofs from the recent results of \cite{CHYZ, HKST2024}.

The usual way to look at the mapping properties of the Wolff potential
is to estimate the Wolff potential pointwise by the potential 
that is an iterated Riesz potential: 
\begin{equation}\label{HV}
\V^f_{\alpha, p}(x) := \I_{\alpha}\Big( (\I_{\alpha}f)^\frac1{p-1} \Big)(x)\,.
\end{equation}
Here,
\begin{equation*}
    \I_{\alpha} f(x) := \int_\Rn \frac{|f(y)|}{|x-y|^{n-\alpha}} \, dy
\end{equation*}
is the Riesz potential of Marcel Riesz.
Viktor P. Havin and Vladimir G. Maz'ya  introduced \eqref{HV} as an independent potential and started its investigation in  their paper \cite{HM} according to \cite[p. 299]{Hedberg1972Zeit}.
Later it has been given the name the Havin--Maz'ya potential.
The  nonlinear potential  had appeared implicitly  in the papers of Norman G. Meyers \cite{Meyers} and Juri Reshentjak \cite{Reshentjak}.
We recall that 
the pointwise estimate 
\begin{equation*}
 \W^{f}_{\alpha ,p}(x) \lesssim \V^f_{\alpha, p}(x)
\end{equation*}
 for all $x\in\Rn$ whenever $p > 1$ and  $0<p\alpha < n$; we refer to 
\cite[Lemma 7.1 p. 136]{HM} and \cite[Theorem 3.1]{AM}. For the validity of
the pointwise estimate
\begin{equation*}
\V^f_{\alpha, p}(x) \lesssim \W^{f}_{\alpha ,p}(x)\,, \quad x\in\Rn\,,
\end{equation*}
whenever  $0<\alpha <n$, $1<p<n$, and $2-\frac{\alpha}{n}<p<\frac{n}{\alpha}$\,,  we refer to
\cite[Theorem 6.1. p. 78]{HM} and \cite[Theorem 3.3]{AM}.

If the function $f$ in \eqref{HV} is $L^q$-integrable over  $\Rn$ with certain values of $q>1$, it is known that the Havin--Maz'ya potential is $L^s$-integrable with
$s=\frac{nq(p-1)}{n-\alpha q p}$, \cite[Theorem 3.1]{Cianchi}.
We extend the  $L^s$-integrability result by considering 
 the $L^s$- integrability in the context of Choquet integrals with respect to the Hausdorff content
$\Ha^\delta_\infty$.
If  $\Omega$ is an open set in $\Rn$, 
we show that the Havin--Maz'ya potential belongs to the space
 \begin{align*}\notag
L^{s}(\Omega, \Ha^{\delta}_\infty) \label{ChoquetSpace}
=
\Big\{f: \Omega \to [0, \infty] :      \int_\Omega |f|^{s} \, d \Ha^\delta_\infty   < \infty          \Big\}
\end{align*}
with certain values of $s$ and $\delta$.
Our main theorem in this context is
Theorem~\ref{thm:H-M>1}.

 If the function $f$ in \eqref{HV} is $L^1$-integrable over $\Rn$, then the Havin--Maz'ya potential is known to be  in a certain  weak $L^s$ space.
 We generalise the known weak $L^s$-result to context of Choquet integrals  in Theorem \ref{thm:H-M=1}.
Here, the weak $L^s$-space  is defined as
\begin{equation}\notag
\wL^s(\Omega, \Ha^{\delta}_\infty)=
\Big\{f: \Omega \to [0, \infty] :          \| f \|_{\wL^s(\Omega, \Ha^\delta_\infty)}     <\infty            \Big\}\,,
\end{equation}
where
\[
\| f \|_{\wL^s(\Omega, \Ha^\delta_\infty)}= \Big(\sup_{\lambda >0} \big( \lambda^s  \, \Ha^{\delta}_\infty(\{x \in \Omega : \vert f(x)\vert>\lambda\}) \big) \Big)^{\frac1s}\,.
\]

In particular, our integrability results for the Havin--Maz'ya potential  in Theorem~\ref{thm:H-M>1} and Theorem~\ref{thm:H-M=1} imply  results  for the Wolff potential.

\begin{theorem}\label{thm:Wolf}
Let $f$ be a  function defined on $\Rn$, $n\geq 2$, and let $\alpha \in (0,n)$  and 
$\kappa\in [0,1)$ be given.\\
 (a) 
 Suppose that $p>1$ and $q>1$ satisfy the inequalities
 $1+1/q-\alpha/n<p$ and $\alpha pq<n$.
  If
 $f\in L^q(\Rn)$,
 then
 \[ 
  \|\W^f_{\alpha, p} \|_{L^s(\Rn,\Ha^{n- \kappa q}_\infty)}
  \le c \|f\|_{L^q(\Rn)}^{\frac{1}{p-1}},
  \]
  where $s:= \frac{ q(p-1)(n-\kappa q)}{n-\alpha pq}$ and $c$ is a  finite positive constant which depends on 
  $\alpha$, $\kappa$, $n$, $p$, and $q$ only.

(b) If $f\in L^{1}(\Rn)$ and
$p \in (2-\frac{\alpha}{n}, \frac{n}{\alpha})$,  then
 \[ 
  \|\W^f_{\alpha, p} \|_{\wL^s(\Rn,\Ha^{n- \kappa}_\infty)}
  \le c \|f\|_{L^1(\Rn)}^{\frac{1}{p-1}},
  \]
where $s:= \frac{ (p-1)(n-\kappa)}{(n-\alpha p)}$ and $c>0$ is a finite constant which depends  on $\alpha, \kappa, n, p$ only.
\end{theorem}

If 
 $q>1$,
 we point out that the exponent $s= \frac{ q(p-1)(n-\kappa q)}{n-\alpha pq}$ is 
sharp. We refer to   Remark~\ref{rem:sharp-2}.

As an application we give new  integrability results  in the context of Choquet integrals to the solutions of the $p$-Laplace equation, $p>1$,
\begin{equation*}
-\operatorname{div}\big( |Du|^{p-2} Du \big)=f
\end{equation*}
 defined on an open set $\Omega$, 
when  a given  function $f$ is having certain integrability. We show  
that  the solution belongs to 
$L^{s}_{\loc}(\Omega, \Ha^{\delta}_\infty)$  with some values of  $s$ depending on the parameters involved, Theorem \ref{thm:application}.

The outline of the paper is as follows: We recall definitions  and give some auxiliary results in Section \ref{sec:pre}.
The maximal operator results are considered in Section \ref{sec:maximal}.
The mapping properties of the Havin--Maz'ya operator are studied in Section \ref{sec:integrability} 
as well as the theorem for the Wolff potential  is proved there. We  study an  application to the $p$-Laplace equation  in Section \ref{sec:pde}.

\section{Preliminaries}\label{sec:pre}

We recall the definition of the  $\delta$-dimensional Hausdorff content for every given set $E$ 
in $\Rn$,  
\cite{Adams1998, Adams2015}
and also  its dyadic counterpart  \cite{YangYuan08}.
An  open ball centred at $x$ with radius $r>0$ is written as $B(x,r)$.

Let $E$ be a set in $\Rn$, $n \ge 1$. Suppose that
$\delta \in (0, n]$.
The $\delta$-dimensional Hausdorff content of $E$ is defined by
\begin{equation}
\Ha_\infty^{\delta} (E) := \inf \bigg\{ \sum_{i=1}^\infty r_i^{\delta}: E \subset \bigcup_{i=1}^\infty B(x_i, r_i)\bigg\}\,,\label{HausdorffC}
\end{equation}
where the infimum is taken over all  
countable (or finite) collections of balls  that cover $E$.
The quantity \eqref{HausdorffC} is also called  the $\delta$-Hausdorff capacity.

If the infimum is taken over all 
countably (or finite) collections of dyadic cubes such that the interior of the union of the cubes covers $E$, the 
dyadic counterpart of $\Ha_{\infty}^{\delta}$ 
is given in \cite{YangYuan08}, that is,
\begin{equation*}
\tilde{\Ha}_\infty^{\delta} (E) := \inf \bigg\{ \sum_{i=1}^\infty \ell (Q_i)^{\delta}: E \subset \inte\big(\bigcup_{i=1}^\infty Q_i \big)\bigg\}\label{HausdorffCD}\,.
\end{equation*}
Here $\ell (Q)$ is the side length of a cube $Q$.

We recall that the Hausdorff content
$\tilde \Ha^\delta_\infty$ 
is a Choquet capacity \cite[Theorem~2.1]{YangYuan08}.  Hence, one of the  benefits of using $\tilde \Ha^\delta_\infty$  is that
for every 
increasing sequence of sets $E_i$  in $\Rn$ the following equality holds,
\begin{equation}\label{BenefitOfDyadicDef}
\lim_{i \to \infty}\tilde \Ha^\delta_\infty(E_i) = \tilde \Ha^\delta_\infty\big(\bigcup_{i=1}^\infty E_i \big).
\end{equation}
The corresponding equality fails for
$\Ha^{\delta}_\infty$.
However,
the Hausdorff content $\Ha^\delta_\infty$ is comparable with its dyadic counterpart  $\tilde \Ha^\delta_\infty$ and the constants in the corresponding inequalities  depend only on $n$. We refer to  \cite[Proposition~2.3]{YangYuan08}.
This means that we have the following lemma.

\begin{lemma}\label{lem:increasing_sequence}
Let  $n\geq 1$ and $\delta \in(0, n]$.  If  
sets  $E_i$ of $\Rn$ form an increasing sequence $(E_i)$,  then
\[
c_1(n) \,  \Ha^\delta_\infty \Big(\bigcup_{i=1}^\infty E_i \Big) \le \lim_{i \to \infty} \Ha^\delta_\infty(E_i) 
\le c_2(n) \,  \Ha^\delta_\infty \Big(\bigcup_{i=1}^\infty E_i \Big)\,,
\]
here  positive constants $c_1$ and $c_2$ depend  only on $n$.
\end {lemma}

We recall the definition of Choquet integral,  with respect to  the  Hausdorff content.
The Choquet integral  was introduced by G.\ Choquet \cite{Cho53} and  applied  by D.\ R.\ Adams  in the study of
nonlinear potential theory  \cite{Adams1988b, Adams1998}.
 If  $\Omega$ is a subset of $\Rn$, $n \ge 1$, and $\delta \in (0, n]$, then
for every   function $f:\Omega\to [- \infty, \infty]$ the integral in the sense of Choquet with respect to  the  $\delta$-dimensional  Hausdorff content is defined by
\begin{equation}\label{IntegralDef}
\int_\Omega |f(x)| \, d  \Ha^\delta_\infty := \int_0^\infty \Ha^\delta_\infty\big(\{x \in \Omega : |f(x)|>t\}\big) \, dt. 
\end{equation}
Since  $\Ha_\infty^{\delta}$ is  a monotone  set function,
the corresponding distribution function
$t \mapsto \Ha^\delta_\infty\big(\{x \in \Omega : |f(x)|>t\}\big)$ 
for each function $f:\Omega\to [-\infty, \infty]$
is decreasing with respect to $t$.
By decreasing property,  the  distribution function 
$t \mapsto \Ha^\delta_\infty\big(\{x \in \Omega : |f(x)|>t\}\big)$ is measurable
 with respect to  the 1-dimensional  Lebesgue measure. Thus, $\int_0^\infty 
 \Ha^\delta_\infty\big(\{x \in \Omega : |f(x)|>t\}\big) \, dt$ 
 is well defined as a Lebesgue integral. 
The right-hand side of \eqref{IntegralDef} can be understood  also as an improper Riemann integral, since a decreasing function is Riemann integrable.
We note that  
\[
\int_0^\infty \Ha^{\delta}_\infty\big(\{x \in \Omega : |f(x)|^p>t\}\big) \, dt = \int_0^\infty p t^{p-1}\Ha^{\delta}_\infty\big(\{x \in \Omega : |f(x)|>t\}\big) \, dt
\]
by a change of variables.

For more  basic properties of Choquet integrals with respect to the Hausdorff content we refer to \cite[Chapter 4]{Adams1998}, \cite[Chapter 2]{HH-S_JFA}, 
\cite[Section 2]{ChenSpector2023},
\cite[Section 2]{HH-S_La}.
Recently,  Choquet integrals with respect to the Hausdorff content  and other set functions have been studied   in many different contexts 
\cite{Den94, CerMS11, PonceSpector2020, PonceSpector2023, 
HH-S_AAG, HH-S_JFA, ChenSpector2023,  HH-S_La,  COS24, HKST2024, Ooi2024, Chen2025, BCRS2025, HH-S_JGA, ChenClaros2025, 
ChenClaros2026, HLLW2026}.

We recall the definitions of Lebesgue spaces in the context of Choquet integral and  Hausdorff content and give a small useful lemma.

Let  $n\geq 1$,  $\delta \in (0,n]$, and $q\in [\delta/n,\infty )$. If $f$ is a non-negative function defined on a set $\Omega$ in $\Rn$,
a Lebesgue space for Choquet integrals with respect to the Hausdorff content is written as
\begin{equation}\notag
\begin{split}
L^{q}(\Omega, \Ha^{\delta}_\infty) &:=
\Big\{f: \Omega \to [0, \infty] :      \int_\Omega |f|^{q} \, d \Ha^\delta_\infty            <\infty            \Big\}. 
\end{split}
\end{equation}
We write that
\[
\|f\|_{L^{q}(\Omega, \Ha^{\delta}_\infty)}:= \Big( \int_\Omega |f|^{q} \, d \Ha^\delta_\infty\Big)^{\frac1q},
\]
and note that it is a quasi-norm.

We define 
weak $L^q$-spaces by using the condition
\begin{equation*}
  \sup_{\lambda >0}\int_{\{x \in \Omega : \vert f(x)\vert>\lambda\}} \lambda^{q}\,d \Ha^{\delta}_\infty
  = \sup_{\lambda >0} \big( \lambda^q  \, \Ha^{\delta}_\infty(\{x \in \Omega : \vert f(x)\vert>\lambda\}) \big)<\infty.
\end{equation*}
We  write 
\[
\| f \|_{\wL^q(\Omega, \Ha^\delta_\infty)}:= \Big(\sup_{\lambda >0} \big( \lambda^q  \, \Ha^{\delta}_\infty(\{x \in \Omega : \vert f(x)\vert>\lambda\}) \big) \Big)^{\frac1q}
\]
and define the weak space as
\[
\wL^q(\Omega, \Ha^{\delta}_\infty):=
\Big\{f: \Omega \to [0, \infty] :          \| f \|_{\wL^q(\Omega, \Ha^\delta_\infty)}     <\infty            \Big\}.
\]
This definition can be found also in \cite[p. 356]{HKST2024}.  Note that $\| \cdot \|_{\wL^q(\Omega, \Ha^\delta_\infty)}$ is a quasi-norm in $\wL^p(\Omega, \Ha^{\delta}_\infty)$. 
We say that a function $u$ belongs to  the space $L^{q}_\loc(\Omega, \Ha^{\delta}_\infty)$,  if 
$u \in L^{q}(O, \Ha^{\delta}_\infty)$
for every open set $O$ compactly inside  $\Omega$.
Similarly, we define the space $\wL^{q}_\loc(\Omega, \Ha^{\delta}_\infty)$.

\begin{lemma}\label{lem:weak_subset}
Let $\Omega$ in  $\Rn$, $n\geq 1$,   be  a bounded set.
If
$\delta \in (0, n]$ and $q \in (\delta/n, \infty)$, then $\wL^q(\Omega, \Ha^{\delta}_\infty) \subset L^{\frac{\delta}{n}}(\Omega, \Ha^{\delta}_\infty) \subset L^1(\Omega, \Ha^{n}_\infty)$. 
\end{lemma}

\begin{proof}
Suppose that $f \in \wL^p(\Omega, \Ha^{\delta}_\infty)$. 
Changing the variables and estimating the integrals, we obtain
\[
\begin{split}
&\int_\Omega |f|^{\frac{\delta}{n}} \, d \Ha^\delta_\infty = \int_0^\infty {\tfrac{\delta}{n}} t^{{\frac{\delta}{n}}-1} \Ha^\delta_\infty \big(\{x \in \Omega : |f(x)|>t \} \big) \, dt\\
&\quad\le \int_0^1 \Ha^\delta_\infty \big(\{x \in \Omega : |f(x)|>t \} \big) \, dt +   
\int_1^\infty {\tfrac{\delta}{n}} t^{{\frac{\delta}{n}}-1-q}   \| f \|_{\wL^q(\Omega, \Ha^\delta_\infty)}^q \, dt\\
&\quad\le \Ha^\delta_\infty (\Omega) +  {\tfrac{\delta}{n}}  \| f \|_{\wL^q(\Omega, \Ha^\delta_\infty)}^q  \frac{1}{q-\delta/n}<\infty.
\end{split}
\]
Hence $\wL^q(\Omega, \Ha^{\delta}_\infty) \subset L^{\frac{\delta}{n}}(\Omega, \Ha^{\delta}_\infty)$.  
Since
\[
\int_\Omega |f| \, d \Ha^n_\infty \le \bigg(\frac{n}{\delta} \bigg)^{\frac1n} \bigg(\int_\Omega |f|^{\frac{\delta}{n}} \, d \Ha^\delta_\infty\bigg)^{\frac{n}{\delta}},
\]
by \cite[Proposition 2.3]{COS24},
 also the second inclusion of the lemma  follows.
\end{proof}

\begin{remark}Throughout the paper  letter $c$ is used as  a generic notion for constants.
The notion $A \lesssim B$ for given quantities $A$ and $B$ means that there exists a finite constant $c>0$ such that $A\le c B$.
\end{remark}

\section{Results for maximal operators and Riesz potentials}\label{sec:maximal}

 Let $n\geq 2$ and $\kappa \in [0, n)$. 
 If   $f$ is locally integrable on $\Rn$, that is
$f \in L^1_\loc(\Rn)$,  we define
\[
\M_\kappa f(x) := \sup_{r>0} \frac{r^\kappa}{|B(x, r)|} \int_{B(x, r)} |f(y)| \, dy,
\]
and  write $\M f := \M_0 f$.

We show that the maximal function 
 $\M_\kappa f $
 belongs to the weak $L^1(\Rn, \Ha^{n-\kappa}_\infty)$ space.
The result is a special case   of \cite[Theorem 7 (ii)]{Adams1998}  by Adams.
As a proof is missing there,
we give a proof for the reader’s convenience.

\begin{theorem}[Theorem 7 (ii) of \cite{Adams1998}]\label{thm:weak_estimate}
Let $\kappa \in [0, n)$. If $f \in L^1(\Rn )$, then for every $\lambda >0$
\[
\Ha^{n- \kappa}_\infty \big( \{x \in \Rn : \M_\kappa f(x) > \lambda\}\big) 
\le \frac{c}{\lambda} \int_{\{x \in \Rn :  |f| > \lambda /4\}} |f(y)| \, dy\,,
\]
where $c$ depends only on $n$.
\end{theorem}

\begin{proof}
Suppose that $f \in L^1(\Rn)$ and $\kappa\in [0,n)$. Let $\lambda >0$ be fixed.
We write
\begin{equation}\label{equ:Alambda}
A_\lambda := \{x \in \Rn : \M_\kappa f(x) > \lambda\}.
\end{equation}
For every $x \in A_\lambda$ we find $r_x >0$ such that
\begin{equation}\label{findradius}
\frac{r_x^\kappa}{|B(x, r_x)|} \int_{B(x, r_x)} |f(y)| \, dy > \lambda.
\end{equation}

We use \eqref{equ:Alambda} and \eqref{findradius} to find a suitable bounded set in order to apply the $5$-covering lemma.
We consider the sets 
$A_\lambda \cap B(0, k)$ for sufficiently large $k>0$.
Let $k$ be given.
We let
$\mathcal{F}$ be a collection of balls with their centres  being in  $ A_\lambda \cap B(0, k)$
and with their radii satisfying the inequality \eqref{findradius}. For these balls,
\[
r_x^{n- \kappa} < \frac{c(n)}{\lambda} \int_{B(x, r_x)} |f(y)| \, dy \le \frac{c(n)}{\lambda} \int_{\Rn} |f(y)| \, dy\,.
\]
Thus, the  values of  radii $r_x$   are  uniformly bounded. 
Hence,  the diameter, $\diam\big(\bigcup_{x \in \mathcal{F}} B(x, r_x)\big) $ is finite.
Now we can apply the $5r$-covering lemma \cite[1.6]{Stein} and obtain pairwise disjoint balls $B(x_i, r_i)$, $i=1, 2, \ldots$\,,  such that
\[
A_\lambda \cap B(0, k) \subset \bigcup_{i=1}^\infty B(x_i, 5r_i). 
\] 
By the definition of the Hausdorff content   and the inequality \eqref{findradius} 
we obtain
\[
\begin{split}
\Ha^{n- \kappa}(A_\lambda \cap B(0, k)) \le \sum_{i=1}^\infty (5r_i)^{n-\kappa} 
\le 5^n \sum_{i=1}^\infty \frac{c(n)}{\lambda} \int_{B(x_i, r_i)} |f(y)| \, dy.
\end{split}
\]
Since the balls $B(x_i, r_i)$ are  pairwise  disjoint,
the properties of the Lebesgue integral
imply that
\[
\Ha^{n- \kappa}_\infty(A_\lambda \cap B(0, k)) \le  
\frac{c(n)}{\lambda} \int_{\bigcup_{i=1}^\infty B(x_i, r_i)} |f(y)| \, dy
\le \frac{c(n)}{\lambda} \int_{\Rn} |f(y)| \, dy.
\]
Thus by Lemma~\ref{lem:increasing_sequence}
\begin{equation}\label{equ:weak_estimate_2}
\Ha^{n- \kappa}_\infty(A_\lambda) \le c(n) \lim_{k \to \infty}\Ha^{n- \kappa}_\infty(A_\lambda \cap B(0, k))
\le \frac{c(n)}{\lambda} \int_{\Rn} |f(y)| \, dy.
\end{equation}

To finish the proof we
 write first that $f_1(x) := |f(x)| \chi_{|f| > \lambda /4}(x)$ and
$f_2(x) := |f(x)| \chi_{0 \le |f| \le \lambda /4}(x)$. Now, $|f(x)|= f_1(x) + f_2(x)$. Since $\M_\kappa$ is a sublinear operator,
$\M_\kappa f(x) \le \M_\kappa f_1(x) + \M_\kappa f_2(x)$. This implies that
\[
A_{\lambda}=
\{\M_\kappa f(x)>\lambda\} \subset \{\M_\kappa f_1(x)>\lambda/2\} \cup \{\M_\kappa f_2(x)>\lambda/2\} \,.
\]
Since $|f_2|\le \lambda/4$, we have
$\{\M_\kappa f_2(x)>\lambda/2\}=\emptyset$.
Hence, 
 we may use 
the estimate \eqref{equ:weak_estimate_2}  
for $f_1$  and obtain
\[
\begin{split}
\Ha^{n- \kappa}_\infty \big(\{\M_\kappa f(x)>\lambda\}\big) 
&=\Ha^{n- \kappa}_\infty (A_{\lambda}) 
\le \Ha^{n- \kappa}_\infty \big(\{\M_\kappa f_1(x)>\lambda/2\}\big) \\
\le  \frac{c}{\lambda} \int_{\Rn} |f_1(y)| \, dy
& = \frac{c}{\lambda} \int_{\{x \in \Rn :  |f| > \lambda /4\}} |f(y)| \, dy. \qedhere
\end{split}
\]
\end{proof}

We recall the following theorem considering the estimates for weak $L^p$ spaces  of 
Naoya Hatano, Ryota Kawasumi,   Hiroki  Saito,  and Hitoshi  Tanaka \cite{HKST2024}.

\begin{theorem}[Theorem 1.1 (i)  of \cite{HKST2024}]\label{thm:weak-to-weak}
Let $\delta \in (0, n]$ and $q \in (\frac{\delta}{n}, \infty)$.
Then there exists a constant $c$ such that
\[
\| \M f \|_{\wL^q(\Rn, \Ha^\delta_\infty)} \le c \| f \|_{\wL^q(\Rn, \Ha^\delta_\infty)}
\]
for all  $f \in  \wL^q(\Rn, \Ha^{\delta}_\infty)$.
\end{theorem}

The next theorem gives mapping properties for the Riesz potential.
The result comes from the paper of
Yuonshou Cao, Long Huang, Dachun Yang, and Ciqiang  Zhuo \cite{CHYZ}.

\begin{theorem}[Theorem 1.9 of \cite{CHYZ}]\label{thm:Huang_et_all}
Let $\alpha \in (0,n)$, $\delta \in (0,n]$, and $\kappa \in [0,\alpha)$.

(a) If $q \in \left(\frac{\delta}{n}, \frac{\delta}{\alpha}\right)$ ,
then there exists a positive constant $c$, depending only on $n$, $\alpha$, $\kappa$, and $\delta$, such that, for every $f \in L^1_{\mathrm{loc}}(\Rn)$,
\[
\| I_\alpha f \|_{L^{\frac{q(\delta-\kappa q)}{\delta-\alpha q}}(\Rn, \Ha^{\delta-\kappa q}_\infty)}
\le c \| f \|_{L^{q}(\Rn, \Ha^\delta_\infty)}.
\]

(b) If $q = \frac{\delta}{n}$, then there exists a positive constant $c$, depending only on $n$, $\alpha$, $\kappa$, and $\delta$, such that, for every $f \in L^1_{\mathrm{loc}}(\Rn)$,
\[
\| I_\alpha f \|_{\wL^{\frac{q(\delta-\kappa q)}{\delta-\alpha q}}(\Rn, \Ha^{\delta-\kappa q}_\infty)}
\le c \| f \|_{L^{q}(\Rn, \Ha^\delta_\infty)}.
\]
\end{theorem}

\begin{proof}
Part (a). The inequality follows from \cite[Theorem 1.9(i)]{CHYZ} whenever we choose both  indexes in the Lorentz space to be equal.

Part (b). The inequality is 
 from  \cite[Theorem 1.9(ii)]{CHYZ}.
\end{proof}

\section{Havin--Maz'ya potential}\label{sec:integrability}

Our first theorem follows  when we iterate 
Theorem~\ref{thm:Huang_et_all}.

\begin{theorem}\label{thm:H-M>1}
 Let  $n\geq 2$,  $\alpha \in (0, n)$,  and 
 $\kappa \in [0, \alpha)$. Suppose that  
 $q>1$   and  $p> 1$ satisfy  the inequalities $p> 1+ 1/q -\alpha/n$  and $\alpha pq < n$.
  If $f$ is a function in $L^{q}(\Rn)$, 
 then
 \[ 
  \|\V^f_{\alpha, p} \|_{L^s(\Rn,\Ha^{n- \kappa q}_\infty)}
  \le c \|f\|_{L^q(\Rn)}^{\frac{1}{p-1}},
  \]
where $s:= \frac{ q(p-1)(n-\kappa q)}{n-\alpha pq}$ and $c$ is a constant independent of $f$.
\end{theorem}

\begin{proof}
Let $f\in L^{q}(\Rn)$ be given.
By Theorem~\ref{thm:Huang_et_all}(a) and \cite[Remark 3.3]{HH-S_JGA}  we have
\[
\| I_\alpha f \|_{L^{\frac{q(n-\kappa q)}{n-\alpha q}}(\Rn, \Ha^{n-\kappa q}_\infty)}
\le c \| f \|_{L^{q}(\Rn, \Ha^n_\infty)}
 = c \| f \|_{L^{q}(\Rn)}
\]
Let us write that $\delta := n- \kappa q$ and $r:=  \frac{q (p-1)(n-\kappa q) }{n- \alpha q}$.
Rewriting
\[
\| I_\alpha f \|_{L^{\frac{p(n-\kappa q)}{n-\alpha} q}(\Rn, \Ha^{n-\kappa p}_\infty)}
= \| (I_\alpha f)^{\frac{1}{p-1}} \|_{L^r (\Rn, \Ha^{\delta}_\infty)}^{p-1}\,,
\]
we obtain
\begin{equation}\label{Rieszr_integ}
\| (I_\alpha f)^{\frac{1}{p-1}} \|_{L^{r}(\Rn, \Ha^{\delta}_\infty)}
\le c \| f \|_{L^{q}(\Rn)}^{\frac{1}{p-1}}.
\end{equation}

Since $pq < \frac{n}{\alpha}$ we have $r < \delta/\alpha$. 
On the other hand,   the inequality $p> 1+ 1/q -\alpha/n$ implies that $r > \delta/n$.
Thus  $r \in (\delta/n, \delta/\alpha)$. 
Integrating 
$\vert ( I_\alpha f)^{\frac{1}{p-1}}\vert $ over a bounded set,
using H\"older's inequality with the exponents $(rn/\delta, nr/(nr-\delta))$,  
and applying 
\cite[Proposition 2.3]{COS24} between  the $n$-dimensional Hausdorff content and $\delta$-dimensional Hausdorff content, and using the inequality
\eqref{Rieszr_integ}, we obtain 
$(I_\alpha f)^{\frac{1}{p-1}} \in L^1_\loc(\Rn)$.
 Now we may apply Theorem~\ref{thm:Huang_et_all}  (a) with $\kappa =0$ and obtain
\begin{equation}\label{rq}
\| \V^f_{\alpha, p} \|_{L^{\frac{r\delta}{\delta-\alpha r}}(\Rn, \Ha^{\delta}_\infty)}
\le c \| (I_\alpha f)^{\frac{1}{p-1}}  \|_{L^{r}(\Rn, \Ha^\delta_\infty)}
\le c \| f \|_{L^{q}(\Rn)}^{\frac{1}{p-1}}.
\end{equation}
 Since $\delta$ and $r$ in \eqref{rq} were chosen as  $\delta=n-\kappa q$ and $r=  \frac{q (p-1)(n-\kappa q) }{n- \alpha q}$, a short calculation gives 
\[
\frac{r \delta}{\delta-\alpha r }= \frac{ q(p-1)(n-\kappa q)}{n-\alpha pq}.
\]
Hence, the estimate 
\eqref{rq} is the desired quasi-norm inequality.
\end{proof}

\begin{remark}\label{rem:sharp_exponent}
(a) 
The classical result
\[
\V^f_{\alpha, p}\in L^{\frac{nq(p-1)}{n-\alpha q p}}(\Rn)
\]
is recovered from Theorem~\ref{thm:H-M>1} with a choice $\kappa =0$.
We refer to \cite[Theorem~3.1 (ii)]{Cianchi} for a proof in Lorentz spaces.
An exact earlier reference seems to be hard to find.

(b) If $\kappa \to \alpha$, then the constant in Theorem~\ref{thm:Huang_et_all} blows up. Thus Theorem ~\ref{thm:H-M>1}
does not  include this endpoint.

(c) The exponent $\frac{q(\delta-\kappa q)}{\delta-\alpha q}$ in Theorem~\ref{thm:H-M>1} is 
sharp. Namely,
an example  \cite[Example 4.4]{CHYZ} shows that the exponent $\frac{q(\delta-\kappa q)}{\delta-\alpha q}$ in 
Theorem~\ref{thm:Huang_et_all}(a) is best possible.
\end{remark}

There are several pointwise estimates for the Riesz potentials
that follow the idea that 
Hedberg used effectively in \cite{Hed72}. 
These estimates as well as norm and quasi-norm estimates derived from them have recently appeared in context of Choquet integrals
in several papers. We refer to \cite{MartinezSpector2021, HH-S_AAG, HH-S_JFA, FuSaSh, MPS2026}.
The following lemma is a modification, where the Lebesgue norm has been replaced by the weak Lebesgue quasi-norm.

\begin{lemma}[Proposition 3.1 of \cite{CHYZ}]\label{lem:Hedberg-2}
Let $\alpha \in (0, n)$, $\kappa \in [0, \alpha)$, and $\delta \in (0, n]$.
If $q \in (\delta/n, \delta/\alpha)$, then there exists a constant $c$  such that
\[
\I_\alpha f(x) \le c  \bigg(M_\kappa f(x)\bigg)^{\frac{\delta-\alpha q }{\delta-\kappa q}}
\| f\|_{\wL^q(\Rn, \Ha^\delta_\infty)}^{\frac{q(\alpha-\kappa)}{\delta -\kappa q}}  
\]
for all $x \in \Rn$  and all $f \in L^1_{\loc}(\Rn)$.
\end{lemma}

\begin{proof}
If we choose  the second index  of the corresponding Lorentz space in \cite[Proposition 3.1 (i)]{CHYZ} to be infinity, the desired  inequality follows.
\end{proof}

\begin{theorem}\label{thm:H-M=1}
 Let  $n\geq 2$,  $\alpha \in (0, n)$, $\kappa \in [0,1)$,  and  $p \in (2-\frac{\alpha}{n}, \frac{n}{\alpha})$ be given and suppose that
$f\in L^{1}(\Rn)$.
 Then
 \[ 
  \|\V^f_{\alpha, p} \|_{\wL^s(\Rn,\Ha^{n- \kappa}_\infty)}
  \le c \|f\|_{L^1(\Rn)}^{\frac{1}{p-1}},
  \]
here $s:= \frac{ (p-1)(n-\kappa)}{(n-\alpha p)}$ and $c$ is a constant independent of $f$.
\end{theorem}

\begin{proof}
Let $f\in L^{1}(\Rn)$ be fixed.
Theorem ~\ref{thm:Huang_et_all} (b) with $\delta =n$  implies  that
\[
\| I_\alpha f \|_{\wL^{\frac{n-\kappa}{n-\alpha}}(\Rn, \Ha^{n-\kappa}_\infty)}
\le c \| f \|_{L^{1}(\Rn, \Ha^n_\infty)}.
\]
Denoting
$\delta := n- \kappa$ and $r:=  \frac{(p-1)(n-\kappa) }{n- \alpha}$
 and recalling \cite[Remark 3.3]{HH-S_JGA}, we obtain
\begin{equation}\label{WeakNormEst}
\| (I_\alpha f)^{\frac{1}{p-1}} \|_{\wL^{r}(\Rn, \Ha^{\delta}_\infty)}
\le c \| f \|_{L^{1}(\Rn, \Ha^n_\infty)}^{\frac{1}{p-1}}
= c \| f \|_{L^{1}(\Rn)}^{\frac{1}{p-1}}.
\end{equation}
Since $p \in (2-\frac{\alpha}{n}, \frac{n}{\alpha})$, we have $r \in (\delta/n, \delta/\alpha)$.
Lemma~\ref{lem:weak_subset} yields  that  $(I_\alpha f)^{\frac{1}{p-1}} \in L^1_\loc(\Rn)$
 by the estimate \eqref{WeakNormEst}. 
Now Lemma~\ref{lem:Hedberg-2}  may be applied with $\kappa =0$ and $q=r$. Hence, the estimate
\eqref{WeakNormEst} shows that
\[
\begin{split}
\V^f_{\alpha, p}(x)^{\frac{ \delta}{\delta-\alpha r }} &\le 
c  \bigg(M \big(\I_\alpha f(x)^{\frac1{p-1}} \big)\bigg)\, 
\| (\I_\alpha f)^{\frac1{p-1}}\|_{\wL^r(\Rn, \Ha^\delta_\infty)}^{ \frac{r \alpha}{\delta-\alpha r }} \\
&\le 
c  \bigg(M \big(\I_\alpha f(x)^{\frac1{p-1}} \big)\bigg)\,
\|  f\|_{L^1(\Rn)}^{ \frac{r \alpha}{(p-1)(\delta-\alpha r) }}. 
\end{split}
\]
In order to obtain a bound to the weak quasi-norm of the potential $\V^f_{\alpha, p}$ we apply 
Theorem~\ref{thm:weak-to-weak}  and the estimate \eqref{WeakNormEst}.  Thus,
\[
\begin{split}
\Big\| \Big(\V^f_{\alpha, p}\Big)^{\frac{\delta}{\delta-\alpha r }}\Big\|_{\wL^r(\Rn, \Ha^\delta_\infty)}
&\le  c \|  f\|_{L^1(\Rn)}^{ \frac{r \alpha}{(p-1)(\delta-\alpha r) }} \| M \big(\I_\alpha f(x)^{\frac1{p-1}} \big)\|_{\wL^r(\Rn, \Ha^\delta_\infty)} \\
&\le   c \|  f\|_{L^1(\Rn)}^{ \frac{r \alpha}{(p-1)(\delta-\alpha r) }} \| \big(\I_\alpha f(x)^{\frac1{p-1}} \big)\|_{\wL^r(\Rn, \Ha^\delta_\infty)} \\
&\le  c \|  f\|_{L^1(\Rn)}^{ \frac{\alpha(n- \kappa)}{(\delta-\alpha r)(n-\alpha) }} \|  f\|_{L^1(\Rn)}^{\frac{1}{p-1 }}\,.
\end{split}
\]
This means that we have
\begin{equation}\label{weakfinal}
\| \V^f_{\alpha, p}\|_{\wL^{\frac{r \delta}{\delta - \alpha r}}(\Rn, \Ha^\delta_\infty)}^{\frac{\delta}{\delta-\alpha r }}
\le 
c 
\|  f\|_{L^1(\Rn)}^{ \frac{ \alpha(n- \kappa)}{(\delta-\alpha r)(n-\alpha) }+\frac{1}{p-1 }}
= c 
\|  f\|_{L^1(\Rn)}^{ \frac{ n- \alpha}{(n- \alpha p) (p-1)}}.
\end{equation}
Since $\delta$ and $r$ were chosen as $\delta =n-\kappa$ and $r=\frac{(p-1)(n-\kappa) }{n- \alpha}$, short calculations give
\[
\frac{r \delta}{\delta-\alpha r }= \frac{ (p-1)(n-\kappa)}{n-\alpha p}
\]
and
\[
\frac{\delta-\alpha r}{\delta}  \frac{ n- \alpha}{(n- \alpha p) (p-1)}  = \frac{1}{p-1}. 
\]
Hence,
the estimate 
\eqref{weakfinal} is the required inequality.
\end{proof}

Since the
Wolff potential can be estimated  pointwise by the Havin--Maz'ya potential, 
both the quasi-norm and the weak quasi-norm estimates follow to the Wolff potential also.

\begin{proof}[Proof of Theorem~\ref{thm:Wolf}]
Since $f$ is a  locally integrable  function,  the mapping
$U \mapsto \int_U |f(y)| \, dy$,  where sets $U$ are Lebesgue measurable,
gives  a non-negative Radon measure.
Thus, the pointwise estimate
 \begin{equation}\notag
\W^{f}_{\alpha ,p}(x) \le c \V^f_{\alpha, p}(x) 
 \end{equation}
 for every $x \in \Rn$, where the constant 
 $c$ depends only on $n$, $p$, and $\alpha$, 
   follows by \cite[Lemma~7.1, p.~136]{HM}, \cite[Theorem 3.1, p.~178]{AM}.
Now Theorem~\ref{thm:H-M=1} and Theorem  \ref{thm:H-M>1} give the desired quasi-norm estimates,
respectively.
\end{proof}

\begin{remark}\label{rem:sharp-2}
 If  $2-\frac{\alpha}{n}<p<\frac{n}{\alpha}$, then 
\[
\W^{f}_{\alpha ,p}(x) \lesssim  \V^f_{\alpha, p}(x) \lesssim \W^{f}_{\alpha ,p}(x)\,.
\]
 For the validity  of these inequalities, we refer to \cite[Lemma 7.1 p. 136]{HM}, \cite[Theorem 3.1]{AM}, \cite[Theorem 6.1. p. 78]{HM}, and \cite[Theorem 3.3]{AM}. 
 Thus,  the exponent $\frac{ q(p-1)(n-\kappa q)}{n-\alpha pq}$ in Theorem~\ref{thm:Wolf}(a) is sharp
by Remark~\ref{rem:sharp_exponent}(c). 
\end{remark}

\section{An application to the PDEs}\label{sec:pde}

Let $\Omega$ in $\Rn$,  $n\geq 2$, be an open set and let $p \in (1, n]$. 
Let us consider functions $u:\Omega\to (-\infty, \infty]$.
If a function $u$ belongs to  the Sobolev space $W^{1,p}_{\loc}(\Omega)$, that is 
$u$ and its  gradient are in $L^p(\Omega)$ locally,
then $u$
is a \emph{weak solution} of the $p$-Laplace equation,
\begin{equation*}
-\operatorname{div}\big( |\nabla u|^{p-2} \nabla u \big)=0\,,
\end{equation*}
if
\[
\int_\Omega |\nabla u|^{p-2} \nabla u \cdot \nabla \phi \, dx =0
\]
for all $\phi \in C^\infty_0(\Omega)$. Every weak solution has a continuous representative. A continuous weak solution is called \emph{$p$-harmonic} in $\Omega$. 
A lower semicontinuous function $u: \Omega \to (-\infty, \infty]$ is called  \emph{$p$-superharmonic} if $u$ is
not identically infinite in each component of $\Omega$, and if for all open  subsets $D$ compactly inside $\Omega$ and all 
$h \in C(\overline D)$ 
 which are  $p$-harmonic in $D$,
the following statement holds:
if  $h\le u$ on  the boundary $\partial D$, then $h\le u$  also  in $D$. We refer to \cite{HKM}.
 Note that the $p$-superharmonic function   $u$ does not necessary belong to $W^{1, p}_\loc(\Omega)$ 
but $u \in L^s(\Omega)$ for $0<s<p-1$, we refer to \cite[Theorems 7.25 and 7.45]{HKM}.

A $p$-superharmonic function satisfies 
the equation
\begin{equation*}
-\operatorname{div}\big( |\nabla u|^{p-2}  \nabla u \big)=f\,,
\end{equation*}
if
\[
\int_\Omega  \lim_{k \to \infty}|\nabla \min\{u, k\}|^{p-2} \nabla \min\{u, k\} \cdot \nabla \phi \, dx = \int_\Omega f \phi \, dx 
\]
for all $\phi \in C^\infty_0(\Omega)$,  we refer to 
\cite[p.~147]{KM}.

The Wolff potential $\W^f_{1, p}$ can be used to  estimate certain
$p$-superharmonic functions pointwise by \cite{KM}.
Theorem~\ref{thm:Wolf} and the result of \cite[Theorem 1.6]{KM} together yield
Theorem \ref{thm:application}.
 We emphasise that Theorem \ref{thm:application} 
 is new whenever $\kappa >0$.  
 The known result is recovered, if $\kappa =0$.

\begin{theorem}\label{thm:application}
Let  $\Omega$ be an open set in $\Rn$, $n\geq 2$, and suppose that
$0\le\kappa <1$,
$1\le q \le n/p$, and $\max\{1, 1 + 1/q -1/n\}<p<n$.
 If a function $f \in L^q (\Omega )$ is given and
$u$ is a non-negative $p$-superharmonic function  defined on  $\Omega$  that satisfies the equation
\begin{equation*}
-\operatorname{div}\big( |Du|^{p-2} Du \big)=f\,,
\end{equation*}
then  we have:
\begin{itemize}
\item[(a)] If $q>1$, then
\[
u \in L^s_\loc(\Omega,\Ha^{n- \kappa q}_\infty).
\]

\item[(b)] If $q=1$, then
\[
u \in \wL^s_\loc(\Omega,\Ha^{n- \kappa}_\infty).
\]
\end{itemize}
Here
\begin{equation*}
s := \frac{ q(p-1)(n-\kappa q)}{n- pq}.
\end{equation*}
\end{theorem}

\begin{proof}
Let $x_0 \in \Omega$ and let $r_0 \in (0, \infty)$ be such that  $B(x_0, 3 r_0) \subset \Omega$.
From \cite[Theorem 1.6]{KM}, we have the pointwise estimate
\[
u(x_0) 
\le c \inf_{x \in B(x_0, r_0)} u(x) 
+ c \W^f_{1, p}(x_0).
\]
Since $u$ is non-negative, the first term on the right-hand side can be controlled locally. 
Hence we need  only to apply Theorem~\ref{thm:Wolf} to the Wolff potential and obtain the part (a) and part (b).
\end{proof}

We conclude this section by remarks.

\begin{remark}
Let $\kappa\in [0,1)$, and let us
suppose that $f\in L^q(\Omega)$ and
$q\in [1, n/p)$. Then 
Theorem~\ref{thm:application} implies that every non-negative $p$-superharmonic  $u$ that satisfies
the equation
\[
-\operatorname{div}(|Du|^{p-2}Du)=f,
\]
is finite outside a set of zero $\Ha^{n-\kappa q}_\infty$-content.
\end{remark}

\begin{remark}
If $\Omega$ is chosen to be the whole $\Rn$
in Theorem~\ref{thm:application}
and $\inf_{x \in \Rn} u(x) =0$, then 
there exist positive, finite  constants $c_i$, $i=1,2$, such that
\[
c_1 \W^f_{1, p}(x_0) \le u(x_0) 
\le c_2 \W^f_{1, p}(x_0),  \mbox{ for all }  x_0\in\Rn,
\]
by \cite[Corollary 4.13]{KM}. If $q>1$, then the exponent 
$\frac{ q(p-1)(n-\kappa q)}{n- pq}$
 in Theorem~\ref{thm:application}(a)  is sharp by Remark~\ref{rem:sharp-2}.
\end{remark}


%
%
%
%
%

\bibliographystyle{amsalpha}

\begin{thebibliography}{HH}




\bibitem{Adams1998}
 Adams, D. R.:
{Choquet Integrals in Potential Theory},
\textit{Publ. Mat.}~ \textbf{42} (1998), {3--66}. 


\bibitem{Adams1988b}
Adams, D. R.: A note on Choquet integrals with respect to Hausdorff capacity. In:
Cwikel, M., Peetre, J., Sagher, Y., Wallin, H. (eds.) Function Spaces and Applica-
tions (Lund 1986), Lecture Notes in Mathematics vol. 1302, Springer, Berlin (1988)
pp.~115--124.


\bibitem{Adams2015}
Adams, D. R.: {\it Morrey Spaces}, {Birkhäuser, Cham--Heidelberg--New York}, 2015.

\bibitem{AH}
Adams,  D.~R., Hedberg,  L.~I.: {\it Function spaces and potential theory}, Grundlehren der mathematischen Wissenschaften, 314, Springer, Berlin, 1996.
 
\bibitem{AM}
 Adams,  D.~R., Meyers, N.~G.:  Thinness and Wiener Criteria for Non-linear Potentials,  \textit{Indiana Univ. Math. J.}~\textbf{22} (1972),  no. 2, 
 169--197.
 


\bibitem{BCRS2025}
Basak, R., Chen, Y.-W. B., Roychowdhury, P., Spector, D.:
The capacitary John-Nirenberg inequality revisited,
\textit{Adv. Calc. Var.} {\bf 18} (2025), no.4, 1361--1385.




\bibitem{CHYZ}
 Cao, Y.,  Huang, L.,   Yang, D., Zhuo, C.:
Sharp Poincar\'e–Sobolev Inequalities of Choquet–Lorentz Integrals
with Respect to Hausdorff Contents on Bounded John Domains,
\textit{J.\ Geom.\ Anal.}~\textbf{36} (2026), article number 218, doi:10.1007/s12220-026-02474-1. 

\bibitem{CerMS11}
Cerd{\'a}, J., Mart{\'i}n, J., Silvestre, P.: Capacitary function spaces.
\textit{Collect.\ Math.}~\textbf{62}  (2011),  no.~1, 95--118.


\bibitem{Chen2025}
Chen, Y.-W.~B.: A self-improving property of Riesz potentials in BMO, \textit{J.\ Geom.\ Anal.}~{\bf 35} (2025), no.~8, Paper No.~237, 23 pp.
%
%
%

\bibitem{ChenClaros2025}
Chen, Y.-W.~B.,  Claros, A.:
$\beta$-dimensional sharp maximal function and its applications. 
arXiv:2407.04456v3.
%

\bibitem{ChenClaros2026}
Chen, Y.-W.~B.,  Claros, A.:
Commutators of fractional integrals with $BMO^{\beta}$ functions. 
arXiv:2602.09742.
%

\bibitem{ChenSpector2023}
Chen, Y.-W., Spector, D.: 
On functions of bounded $\beta$-dimensional mean oscillation,
\textit{Adv. Calc. Var.}~\textbf{17} (2024), 975--996.

%
%






\bibitem{COS24}
 Chen, Y.-W.\ B.,  Ooi, K.\ H., Spector, D.: Capacitary maximal inequalities and applications, \textit{J.\ Funct.\ Anal.}~\textbf{286} (2024), no.\ 12, article number 110396, doi:10.1016/j.jfa.2024.110396.

\bibitem{Cho53}
 Choquet, G.: Theory of capacities. \emph{Ann. Inst. Fourier (Grenoble)}~\textbf{5} (1953--1954),  13--295. 
 
\bibitem{Cianchi}
 Cianchi,  A.: Nonlinear potentials, local solutions to elliptic equations, and rearrangements, \textit{Ann.\ Sc.\ Norm.\ Super.\
Pisa}~\textbf{10} (2011), 335--361.

%
\bibitem{Den94}
Denneberg, D.: {\it Non-additive measure and integral}, Theory and Decision Library 
Series B: Mathematical and Statistical Methods vol. 27,  Kluwer Academic Publishers Group, Dordrecht, 1994.
%


\bibitem{FuSaSh}
Futamura, T., Sawano, Y., Shimomura T.:
Weak-type estimates for variable Riesz potentials with respect to Hausdorff content over metric measure spaces,
\textit{Georgian Math. J.}\textbf{33} (2026), no. 2, 289--299.



\bibitem{HH-S_AAG}
Harjulehto, P., Hurri-Syrj\"anen, R.: 
Estimates for the variable order Riesz potential with application. 
In: 
Lenhart, S.,  Xiao, J. (eds.),
\textit{Potentials and Partial Differential Equations: The Legacy of David R. Adams},   
Advances in Analysis and Geometry vol. 8, De Gruyter, Berlin  (2023) pp. 127--155.






\bibitem{HH-S_JFA}
Harjulehto, P., Hurri-Syrj\"anen, R.: {On Choquet integrals and Poincar\'e-Sobolev  type inequalities},  
\textit{J.\ Funct.\ Anal.}~\textbf{284} (2023), issue 9, article number 109862, doi:10.1007/s44007-024-00131-z.




\bibitem{HH-S_La}
Harjulehto, P., Hurri-Syrj\"anen, R.: 
On Choquet integrals and Sobolev type inequalities, 
\textit{La Matematica} \textbf{3} (2024), 1379--1399.
%




\bibitem{HH-S_JGA}
Harjulehto, P., Hurri-Syrj\"anen, R.: On Lebesgue points and measurability with Choquet integrals, \textit{J.\ Geom.\ Anal.}~\textbf{36} (2026), article number 194, doi:10.1007/s12220-026-02419-8.









%





%
\bibitem{HKST2024}
Hatano, N., Kawasumi, R., Saito H., Tanaka, H.:
{Choquet integrals, Hausdorff content and fractional operators,}
\textit{Bull. Austr. Math. Soc.}
\textbf{110} (2024),
Issue 2,
355--366.
\textit{
Correction} to 'Choquet integrals, Hausdorff content and fractional operators' in
\emph{Bull. Austr. Math. Soc.}
\textbf{111} (2025),
Issue 3,
568--570.





\bibitem{Hed72}
 Hedberg, L.\ I.:
On certain convolution inequalities, \textit{Proc.\ Amer.\ Math.\ Soc.} \textbf{36} (1972), 505--510.


\bibitem{Hedberg1972Zeit}
Hedberg, L. I.:
Non-linear Potentials and Approximation in the mean by Analytic Functions,
\textit{Math. Z.}\textbf{129}(1972), 299--319.




\bibitem{HedbergWolff1983}
Hedberg, L. I., Wolff, Th. H..:
Thin sets in nonlinear potential theory,
\textit{Ann. Inst. Fourier (Grenoble)} \textbf{33} (1983),  no. 4, 161--187.


\bibitem{HKM} 
 Heinonen, J.,   Kilpel\"ainen, T., Martio, O.: {\it Nonlinear Potential Theory of Degenerate Elliptic Equations}, Oxford University Press, Oxford, 1993.




\bibitem{HLLW2026}
Huang, M., Lahti, P., Li, J., Wang, Z.:
Boxing inequalities for relative fractional perimeter and fractional Poincar\'e-type inequalities on John domains with the BBM factor.
arXiv:2604.21506v1.






\bibitem{KM}
Kilpeläinen, T., Mal{\'y}, J.: The Wiener test and potential estimates
for quasilinear elliptic equations, \textit{Acta Math.}~\textbf{172} (1994), 137--161.

\bibitem{MartinezSpector2021}
Mart\'{i}nez, \'{A}.\ D.,  Spector, D.:
An improvement to the John-Nirenberg inequality for functions in critical Sobolev spaces,
\textit{Adv. Nonlinear Anal.}~\textbf{10} (2021), 877--894.






\bibitem{HM} 
 Maz'ya, V.\ G.,  Havin, V.\ P.: 
A nonlinear potential theory, 
\textit{Russian Math.\ Surveys}~\textbf{27} (1972), no.\ 6, 71--148.


\bibitem{Meyers}
Meyers, N. G.:
A theory of capacities for potentials of functions in Lebesgue classes,
\textit{Math. Scand.} \textbf{26} (1970), 255-292.



\bibitem{MPS2026}
Mihula, Z., Pick, L., Spector D.:
Potential trace inequalities via a Calderon-type theorem,
\textit{J. London Math. Soc.} (2)\textbf{113} (2026)no. 3, Paper No.  e70504, 35 pp.






\bibitem{Ooi2024}
Ooi, K.\ H.:
On the dual of Choquet integrals spaces associated with capacities,
\textit{Studia Math.}\textbf{274} (2024), no. 3, 249--268.



\bibitem{OoiPhuc2022}
Ooi, K. H.,  Phuc, N. C.:
The Hardy-Littlewood maximal function, Choquet integrals, and embeddings of Sobolev type,
\textit{Math. Ann.}~\textbf{382} (2022), 1865--1879.

 

 
\bibitem{PonceSpector2020}
Ponce, A.G., Spector, D.:
A boxing inequality for the fractional perimeter.
\textit{Ann. Sc. Norm. Super. Pisa Cl. Sci.}(5)
\textbf{20} (2020), 107--141.




\bibitem{PonceSpector2023}
Ponce, A. G.,  Spector, D.: 
Some remarks on Capacitary Integrals and Measure Theory.
In: 
Lenhart, S.,  Xiao, J. (eds.)
\textit{Potentials and Partial Differential Equations: The Legacy of David R. Adams}, 
Advances in Analysis and Geometry vol. 8, De Gruyter, Berlin  (2023)  pp. 127--155.

\bibitem{Reshentjak}
Reshentjak, Ju. G.: 
On the concept of capacity in the theory of functions wth generalised derivatives.
\textit{Sib. Mat. Zh.} \textbf{10} (1969), 1109--1138, (Russian).
English translation: Siberian Mat. J., 10 (1969), 818--842.



\bibitem{Stein}
Stein, E. M.:
{\it Singular Integrals and Differentiability Properties of Functions},
Princeton University Press, Princeton, New Jersey, 1970.


\bibitem{YangYuan08}
Yang D., Yuan, W.:  A note on dyadic Hausdorff capacities. \textit{Bull.\ Sci.\ Math.}~\textbf{132} (2008), 500--509.







\end{thebibliography}

\end{document}